\documentclass{revtex4}

\usepackage{latexsym}
\usepackage{graphics}
\usepackage{amsmath}
\usepackage{amssymb}
\usepackage{gensymb}
\usepackage[usenames,dvipsnames]{color}
\usepackage{epstopdf}
\usepackage{bm}
\usepackage{amsthm}
\usepackage{graphicx}
\usepackage{comment}
\usepackage{xcolor}

\usepackage[T1]{fontenc}
\usepackage{fourier}
\usepackage{subcaption}
\usepackage{dsfont}
\usepackage{amsthm}
\usepackage[utf8]{inputenc}

\newtheorem{theorem}{Theorem}
\newtheorem{lemma}{Lemma}

\begin{document}
\title{Convergence of Rain Process Models to Point Processes}

\author{Scott Hottovy}
\affiliation{Department of Mathematics, United States Naval Academy, Annapolis, Maryland}
\author{Samuel N. Stechmann}
\affiliation{Department of Mathematics, and Department of Atmospheric and Oceanic Sciences, University of Wisconsin–Madison,
Madison, Wisconsin}



\begin{abstract}
A moisture process with dynamics that switch after hitting a threshold gives rise to a 
rainfall process. This rainfall process is characterized by its 
random holding times for dry and wet periods. On average, the holding times for the wet periods are much shorter than the dry. Here convergence is shown for the rain fall process to a point 
process that is a spike train. The underlying moisture process for the point process is a 
threshold model with a teleporting boundary condition. This approximation allows simplification of the model with many exact formulas for statistics. The convergence is shown by a Fokker-Planck derivation, convergence in mean-square with respect to continuous functions, of the moisture process, and convergence
in mean-square with respect to generalized functions, of the rain process.
\end{abstract}

\maketitle

\section{Introduction}

Rain processes are often studied and tested as single column models in the atmosphere \cite{bm86,fmp04}. These 
single column models are used to ``trigger'' a transition to convection or rainfall. Furthermore, these rainfall models are simple and have uses in global climate models (GCMs) \cite{ln00,sz14,bkj17}. One type of rainfall model that has been successful  is a Renewal process (\cite{bmr07,fl87,g64,svb98}). A renewal process is a continuous-time Markov process that is 
defined by its holding times \cite{c62}. For example, a simple rain renewal process, $\sigma(t)$, 
may be defined by its holding times $\tau^d$ when $\sigma(t)=0$, defined to be that the column of air is
dry, and $\tau^r$ when $\sigma(t) = 1$, defined to be that the column of air is raining. Thus a dry (or rain) event will last $\tau^d$ (or $\tau^r$) time where the distribution of the (possibly) random
times are given.

The model presented here was previously studied in \cite{sn11,sn14,hs15siap,asn16}.  
The underlying process is a one dimensional continuous-time stochastic process 
modeling the moisture $q(t)\in(-\infty,\infty)$ for $t\geq0$, typically in 
cm, for a parcel of 
air. This quantity is vertically integrated and averaged over a square domain to 
give the units of length. Here the rain process is modeled as $\sigma(t)$ an indicator function for rain. One choice of process for $q(t)$ is the hysteresis dynamics
studied in \cite{hs15siap}. There the $q(t)$ process is governed 
by the stochastic differential equations (SDE)
\begin{equation}
    \label{eq:D2}
    dq(t) = \left \{\begin{array}{lr}
    m\;dt + D_0 \;dW_t & \mbox{for }\sigma(t) = 0 \\
    -r\;dt + D_1 \;dW_t & \mbox{for }\sigma(t) = 1 
    \end{array}\right . , \quad q(0) = 0, \quad \sigma(0) = 0,
\end{equation}
where $m$ and $r$ are the moistening and rain rates respectively, and 
$D_0$ and $D_1$ are the fluctuations of moisture during the respective
states. The dynamics of $\sigma(t)$ switch from 0 to 1 when $q(t)$ reaches
a fixed threshold $b$. That is $(q(0),\sigma(0))=(0,0)$ and 
$\sigma(t) = 1$ for $t=\inf\{t\geq0, q(t)=b\}$. Then $\sigma(t)$ switches
back to zero when $q(t)$ reaches the threshold $q(t)=0$. This 
model is referred to as the deterministic trigger, two threshold (D2) model in \cite{hs15siap}. A realization of the processes $q(t)$ and $\sigma(t)$ are shown in Figure \ref{fig:D2}. In a renewal process perspective, $\sigma(t)$ is modeled by the holding times $\tau_d$ and $\tau_r$. These times 
are the first passage times for Brownian motion with drift to travel 
a distance of $b$ units. 

Models of triggering precipitation are used for convective parameterizations. Convective parameterizations are useful in global climate and circulation 
models (GCMs). In GCMs, the models are run for long times to 
examine potential effects, for example, of climate change. In long times, 
the model \eqref{eq:D2} was studied in \cite{asn16}. In a long time view,
the dry events dominate the rain events. In Figure 2 the top panel shows 
$\sigma(t)$ process is plotted up to time $T=...$, and in the bottom panel
an example of the rain process in observations (OBS DETAILS). Thus for 
long times, $r$ is a much larger rate than $m$, and $\sigma(t)$ resembles 
a point process.

\begin{figure}
	\begin{center}
			\includegraphics[scale=0.25,angle=0]{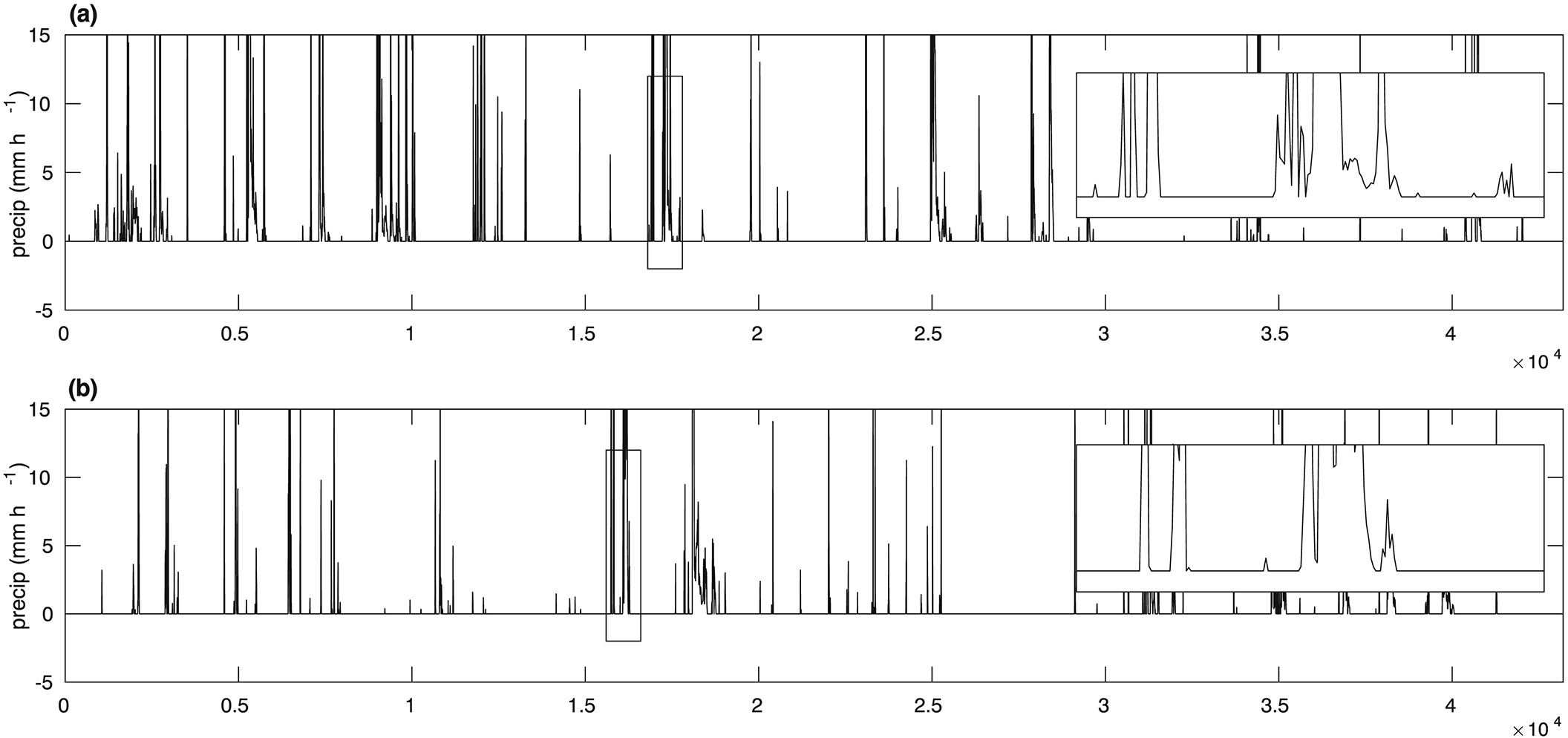}
			
			\includegraphics[trim=3.25cm 19cm 3.25cm 2cm,clip,width=1\textwidth]{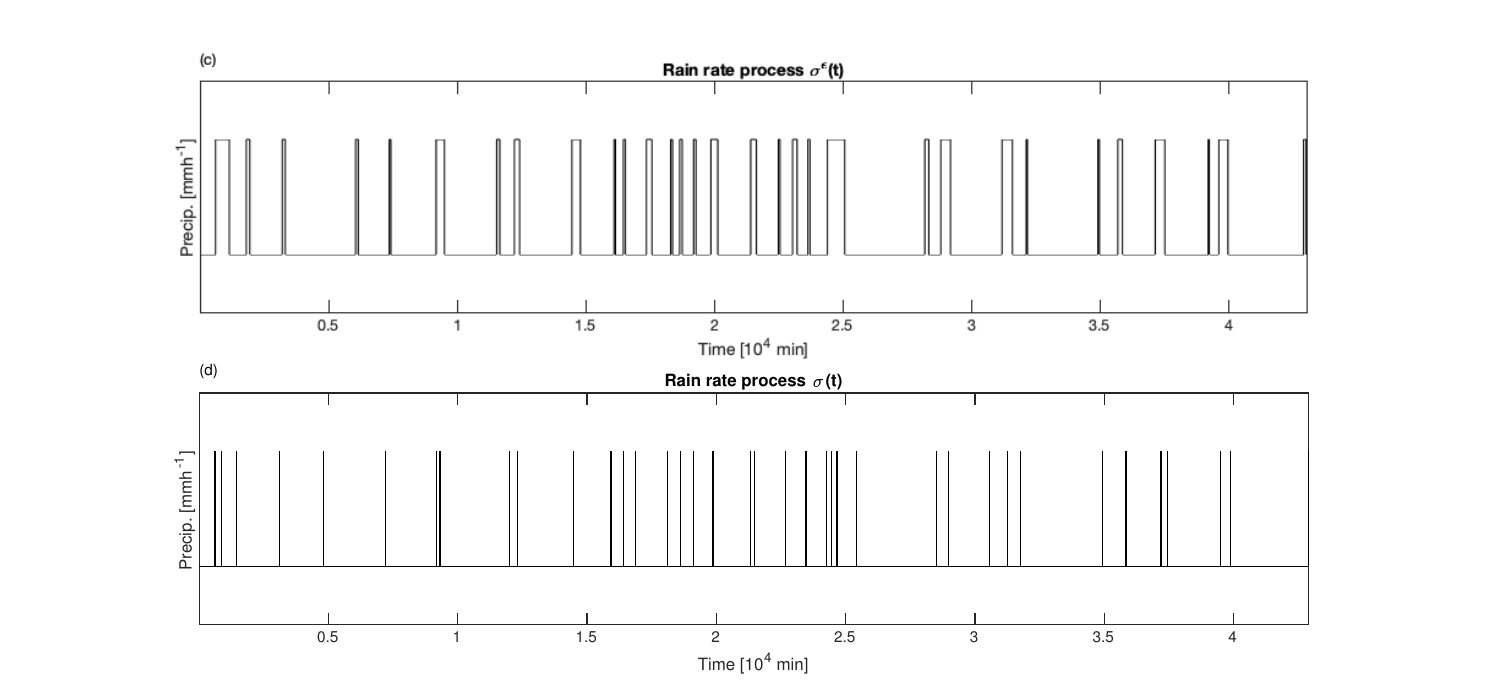}
	\end{center}
	\caption{Sample precipitation time series from observations at (a) Manus Island and (b) Nauru Island reproduced from Fig 3. of \cite{asn16} with permission from the authors. Panels 
	(c) and (d) are realizations of the model of the rain rate process $\sigma^\epsilon(t)$ with finite rain rate $r$ and $\sigma(t)$ the point process, respectively.}
		\label{fig:D2}

\end{figure}

The main result of the paper is to define and show convergence of the threshold model above as $r\to\infty$. For example, 
on the level of renewal processes, $\tau^r\to0$ and thus
$\sigma(t)$ converges to a process that is zero everywhere and
has spikes at infinity at random times $\tau^d$. However, $\sigma(t)$ is right continuous and has left hand limits, where as the spike train is not. Thus the mode of convergence is not clear. For $q(t)$ the limit is also unclear, but will be redefined in a way to show convergence with respect to the topology on continuous functions with the uniform metric. In this study, the limiting processes are define (in Section \ref{sec:models}) and convergence is shown both heuristically (for the Fokker-Planck equation) and rigorously.

There are many novel aspects of this work. The limit jump process $q(t)$ has an associated Fokker-Planck equation that 
is derived using a matched asymptotic method. The resulting Fokker-Planck equation has a peculiar boundary flux condition which defines a ``teleporting'' boundary condition of $q(t)$.
The processes are decoupled into evaporating and
precipitating processes. Only after this decoupling can convergence of the evaporation 
processes be shown rigorously with respect to the uniform metric 
on the space of continuous functions. Finally, the rain process $\sigma(t)$ is shown to converge  
rigorously with respect to the the generalized function space. This proof shows convergence of a renewal process to a delta process. Further more, the proof shows what kinds of bounds 
the rain event times $\tau^{r}$ need in order 
for integrated convergence to hold. 

The convergence results shown here have the potential to 
impact various other fields. Many fields of study use 
similar renewal processes to model phenomena. The connections 
to rain models was made above. In addition, there has been 
much work in queuing theory to approximate point processes 
with renewal processes (e.g. \cite{w82,b94}), and using threshold triggers in financial models \cite{lp19}. The strongest connection is with neuron stochastic integrate and fire models (see \cite{sg13} for a review). The moisture process with a finite rain rate is similar to a Wiener Process model of a single Neuron with refractoriness. A similar model was studied in \cite{aetal08} where the refractory time was constant. Here, 
the refractory time is random and coincides with the rain duration time $\tau^{r}$. Thus the work here is applicable to understanding the differences in using a model without refractoriness versus a model with a short, possible random, refractory time. 
The structure of the paper is as follows. The processes for moisture and rain are defined in Section \ref{sec:models}. The modes of convergence are discussed in 
Section \ref{sec:convergence}. The heuristic derivation of the Fokker-Planck equation
is shown in Section \ref{sec:FP}. Rigorous convergence of the moistening process
$E^\epsilon$ to $E$ is shown with respect to $L^2$ in Section \ref{sec:PathConvergence} and the rain process $\sigma^\epsilon$ is shown to converge to the sum of delta 
distributions $\sigma$ with respect to generalized functions in Section \ref{sec:MeanSquare}. The results are summarized in Section \ref{sec:conclusions}.

\section{Model Description}
\label{sec:models}

In this section the moisture and precipitation processes are defined. First the underlying 
moisture process of the renewal rain process is defined. The processes are defined with 
a small parameter $\epsilon$ with the limit as $\epsilon\to0$ in mind. 

The moisture 
process $q^\epsilon(t)\in\mathbb{R}$ is defined as the solution to the stochastic 
differential equation (SDE), 
\begin{equation}
    \label{eq:D2}
    dq^\epsilon(t) = \left \{\begin{array}{lr}
    m\;dt + D_0 \;dW_t & \mbox{for }\sigma^\epsilon(t) = 0 \\
    -\frac{r}{\epsilon}\;dt + D_1 \;dW_t & \mbox{for }\sigma^\epsilon(t) = 1 
    \end{array}\right . , \quad q(0) = 0, \quad \sigma^\epsilon(0) = 0,
\end{equation}
where $m$ and $r/\epsilon$ are the moistening and rain rates, and 
$0<D_0\leq D_1$ are the fluctuations of moisture during the respective
states. The rain process, $\sigma^\epsilon(t)\in\{0,r/\epsilon\}$ are as follows: since $\sigma^\epsilon(0)=0$, let $\tau_1^{d,\epsilon} \equiv \inf\{t>0| q^{\epsilon}(t) = b\}$. Then $\sigma^\epsilon(t)=0$ for $t\in[0,\tau_1^d)$. Next let $\tau_1^r \equiv \inf\{t>\tau_1^d|q^\epsilon(t)=0\}$, and $\sigma^\epsilon(t)=r/\epsilon$ for $t\in[\tau_1^d,\tau_1^r)$. This 
process repeats up to an arbitrary final time $T$.

The associated processes, as $\epsilon\to0$, are defined as $q(t)$ and $\sigma(t)$ for 
the moisture and rain processes. The moisture process is the solution to the SDE, 
\begin{equation}
    \label{eq:qeq}
    dq(t) = m\;dt + D_0\;dW_t, \quad q<b, \quad q(0) = 0, 
\end{equation}
with the unusual boundary condition as follows: Let the usual stopping time be as follows $\tau_1^d=\inf\{t>0|q(t)=b\}$. Then at time $t>\tau_1^d$ the process $q(t)$ jumps or ``teleports'' to $q=0$. Thus 
\begin{equation}
    \lim_{t\rightarrow (\tau_1^d)^-} q(t) = b, \quad \lim_{t\rightarrow (\tau_1^d)^+} q(t) = 0, \quad q(\tau_1^d) = b. 
\end{equation}
Then the process starts over using the dynamics of \eqref{eq:qeq} until 
$\tau_2^d=\inf\{t>\tau_1^d|q(t)=b\}$, and the process repeats. The stopping
times $\tau_i^d$ are the dry event times of the process. 
These dynamics arise from the heuristic Fokker-Planck derivation in the 
next section (see Section \ref{sec:FP}). 

Examples of the processes are shown in Figure \ref{fig:qExamples}. The processes with finite rain rate $r/\epsilon$ for $\epsilon>0$ are shown in panels (a) and (b). Panel (a) is the 
moisture process $q^\epsilon(t)$ defined is equation~\eqref{eq:D2}. The rain rate process is shown in (d) and takes value $r/\epsilon$ when $q^{\epsilon}(t)$ reaches level $b$ for the first time (panel (a) in black) and resets to 
zero when $q^\epsilon(t)$ reaches zero (panel (a) in gray). This process repeats. The limiting processes are shown in panels (c) and (d). Panel (c) shows the limiting moisture process $q(t)$ defined in equation \eqref{eq:qeq} and panel (d) shows the rain process defined in equation \eqref{eq:sigma}. The moisture process is a Brownian motion with positive drift until 
reaching level $b$. When $q(t)=0$, the process $\sigma(t)$ takes an infinite value and the moisture process is reset at zero. 

\begin{figure}
	\begin{center}
	\includegraphics[clip,width=.75\textwidth]{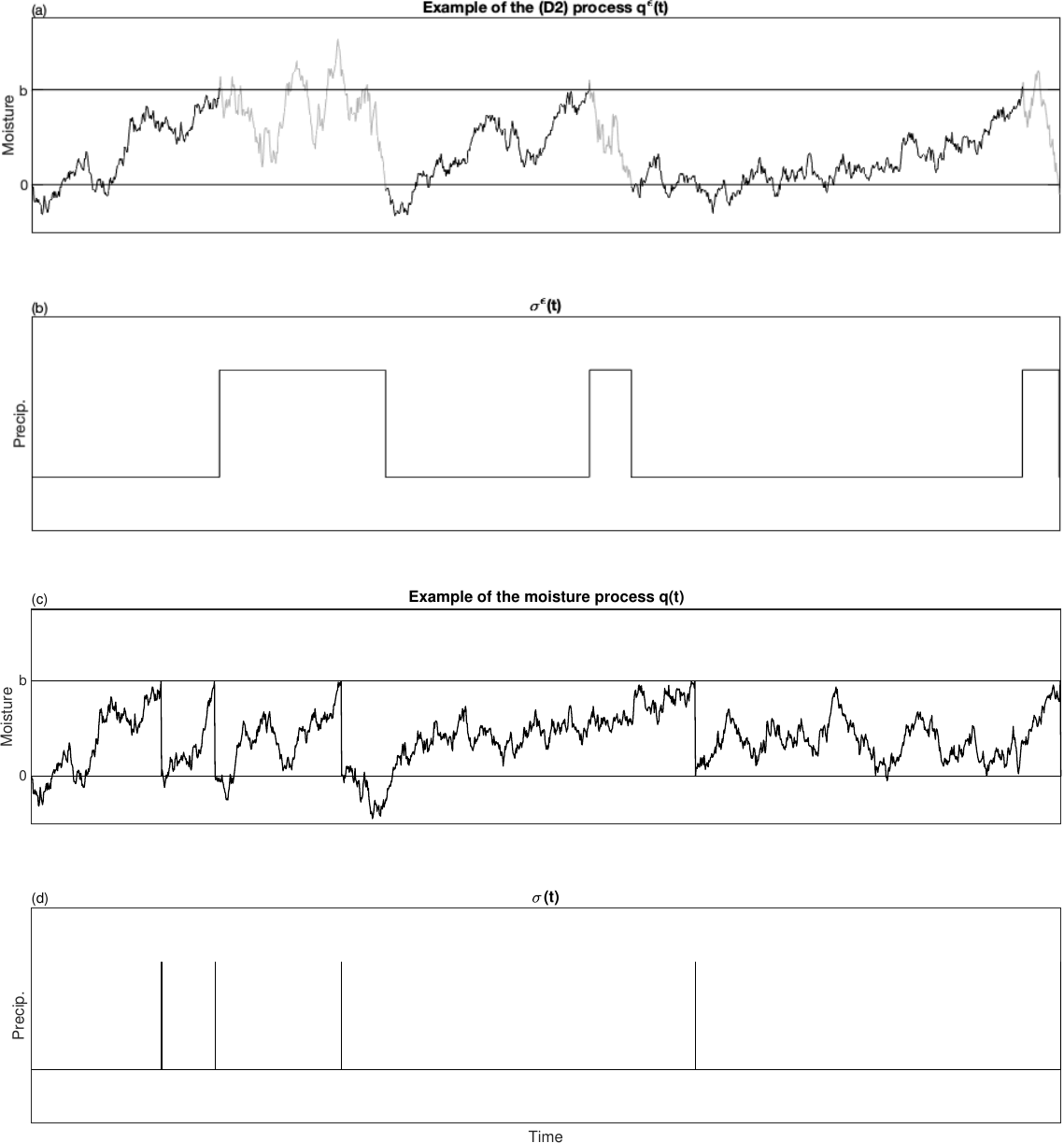}
	\end{center}
	\caption{Realizations are plotted of the processes (a) $q^\epsilon(t)$, (b) $\sigma^{\epsilon}(t)$ for rain rate $r/\epsilon$ defined in equation \eqref{eq:D2} with $\epsilon>0$ and (c) $q(t)$ and (d) $\sigma(t)$ defined in equation \eqref{eq:qeq} and \eqref{eq:sigma} respectively.}
		\label{fig:qExamples}
\end{figure}

From the definition of $\tau_i^d$ above, the rain point process $\sigma(t)$ 
is defined as
\begin{equation}
    \label{eq:sigma}
    \sigma(t) = b\sum_{i=1}^{\mathcal{N}(T)} \delta(t-\tau_i^d),
\end{equation}
where $\mathcal{N}(T)$ is the random variable of the number of times the process $q(t)$ reaches $b$ in time $T$. The quantity $b$ arises because the moisture process $q^\epsilon$ loses moisture at a rate of $r/\epsilon$ per time, on average. The moisture process $q(t)$ loses all the moisture built up ($b$) instantaneously. 

Note that $q^\epsilon(t)$ has continuous paths while $q(t)$ 
has jump discontinuities. Thus any mode of convergence between $q^\epsilon$ and $q$ with an associated metric (e.g. uniform or Skorohod) will fail. There
is another way to define both $q^\epsilon$ and $q$ in which convergence 
with respect to $L^2$ with the uniform metric on the space of continuous functions ($\mathcal{C}[0,T]$) can be shown. To do so, $q^\epsilon(t)$ is decomposed into an evaporation process, $E^\epsilon(t)$, and precipitating
process $P^\epsilon(t)$. These processes are defined as, 
\begin{equation}
\label{eq:EPdef}
	dE_t^\epsilon = \left \{ \begin{array}{lr}
	m\;dt + D_0dW_t & \mbox{for }\sigma_t^\epsilon=0\\
	0 & \mbox{for }\sigma_t^\epsilon=1
	\end{array} \right . ,
	\quad \mbox{ and }\quad
	dP_t^\epsilon = \left \{ \begin{array}{lr}
	0 & \mbox{for }\sigma_t^\epsilon=0 \\
		-\frac{r}{\epsilon}\;dt + D_0dW_t & \mbox{for }\sigma_t^\epsilon=1
	\end{array} \right . .
\end{equation}
Thus the moisture process $q^\epsilon(t)$ is written as 
$$ q^\epsilon(t) = E^\epsilon(t)+P^\epsilon(t). $$
In the limit, the jumps will be captured in the $P^\epsilon$ process. In the following section it will be shown (see Section \ref{sec:PathConvergence}) that $E^\epsilon\rightarrow E$, where 
$E(t)$ is defined as the solution to the SDE
\begin{equation}
    \label{eq:Edef}
    dE(t) = m\;dt + D_0\;dW_t, \quad E(0) = 0. 
\end{equation}
Furthermore, the $\sigma$ process is defined as above in \eqref{eq:sigma} 
where $\tau_i^d=\inf\{t>0|E(t) = kb, \; k\in\mathbb{N}\}$, i.e. the first passage time of Brownian motion with drift to $kb$.

\section{Convergence to a Point Process}
\label{sec:convergence}

In this section convergence is shown both heuristically (e.g. 
Section \ref{sec:FP}) and rigorously (e.g. Sections \ref{sec:MeanSquare} and \ref{sec:PathConvergence}). 

Note that the simplest ideas of convergence break down when considering path-wise convergence of $q^\epsilon$ to $q$ and $\sigma^\epsilon$ to $\sigma$. This is because $q^\epsilon$ is a continuous process ($\sigma^\epsilon$ is left continuous and has right hand limits) for all $\epsilon>0$ and $q$ is a process with jumps ($\sigma$ no longer is left continuous). Thus, there is no topology with associated metric $d$ such that $q^\epsilon\rightarrow q$ with respect to $d$. However, one can show that $q^\epsilon$ converges in a notion weaker than the Skorohod topology. See \cite{kurtz91} for these conditions. This convergence happens in a topology which does not have an associated metric (see \cite{j97}). This is not done here as it is technical and does not give any insight to the model or approximation. 

The following 
three subsections prove convergence of the various processes introduced in Section~\ref{sec:models}.
 In Section \ref{sec:FP} the 
Fokker-Planck equation for $q^\epsilon$ is shown to converge (formally) to a Fokker-Planck equation for $q$. This derivation
gives rise to an interesting partial differential equation (PDE) with unusual ``teleporting'' boundary conditions. 
 In Section \ref{sec:PathConvergence} convergence in paths is shown for $E^\epsilon$ to $E$ with respect
to the uniform metric for continuous functions on $[0,T]$. In Section \ref{sec:MeanSquare} convergence is shown for $\sigma^\epsilon$ to 
$\sigma$ with respect to generalized functions. This norm is necessary because $\sigma$ is a sum of dirac delta functions. In addition, this convergence 
is the ``natural'' convergence to consider when analyzing the errors between using $\sigma^\epsilon$ and a point process ($\sigma$) in, for example, a GCM. 

\subsection{Fokker-Planck Equation}
\label{sec:FP}
In this section, the derivation for the Fokker-Planck equation of equation \eqref{eq:qeq} is shown. 

    The Fokker-Planck equation for the (D2) process (see \cite{hs15siap}) is composed of two densities. These densities are denoted $\rho_0$ and $\rho_1$ for the dry and rain states respectively. These densities follow the following Fokker-Planck equations 
\begin{align}
	\label{eq:rho0}
	\partial_t \rho_0 =& -m\partial_q \rho_0 + \left .\frac{D_0^2}{2}\partial_q^2 \rho_0 -\delta(q)f_1\right |_{q=0}, \quad -\infty<q<b, t\geq0, \\
		\label{eq:rho1}
	\partial_t \rho_1 =& \frac{r}{\epsilon}\partial_q \rho_1 + \left .\frac{D_1^2}{2}\partial_q^2 \rho_1 +\delta(q-b)f_0\right |_{q=b}, \quad 0<q<\infty, t\geq0, 
\end{align}
where the fluxes $f_i$ are defined as
\begin{subequations}
\begin{align}
    \label{eq:fluxes}
    f_0(q,t) &= -m\rho_0(q,t) - \frac{D_0^2}{2}\partial_q\rho_0(q,t) \\
    f_1(q,t) &= \frac{r}{\epsilon}\rho_1(q,t) - \frac{D_1^2}{2}\partial_q\rho_1(q,t)m,
\end{align}
\end{subequations}
and with the following conditions, 
\begin{align}
	\rho_0(b)=\rho_1(0) &= 0 \\
	\int_{-\infty}^\infty \rho_0(q,t)+\rho_1(q,t) \;dq &= 1, \quad t\geq0. 
\end{align}

The proposed limit as $\epsilon\to0$ for the Fokker-Planck equation is
\begin{align}
	\label{eq:limitrho0}
	\partial_t \rho_0 =& -m\partial_q \rho_0 + \left .\frac{D_0^2}{2}\partial_q^2 \rho_0 +\delta(q)f_0\right |_{q=b}, \quad -\infty<q<b, t\geq0, \\
	\label{eq:limitrho1}
	\rho_1 = &0. 
\end{align}
with the following conditions, 
\begin{align}
	\rho_0(b) &=0\\
	\int_{-\infty}^b \rho_0(q,t) \;dq &= 1. 
\end{align}

The analysis follows 
asymptotic matching conditions from \cite{bo13} (Chpt. 9). Consider two regions $[0,\epsilon]$ and $[\epsilon,\infty)$. Let $\rho_{1,B}$ be the density in the first, boundary layer region. For this equation, define the rescaled variable $\tilde{q} = \frac{1}{\epsilon}q$. This yields the equation
\begin{equation}
\label{eq:BL}
  \partial_t \rho_{1,B} = \frac{r}{\epsilon^2}\partial_{\tilde{q}}\rho_{1,B} +\frac{D_1^2}{2\epsilon^2}\partial_{\tilde{q}}^2 \rho_{1,B}
 \end{equation}
Let $\rho_{1,B}$ have the asymptotic expansion of the form, 
$$ \rho_{1,B} = \rho_{1,B}^0 + \epsilon \rho_{1,B}^1 + O(\epsilon^2).$$
Substituting this expansion into equation \eqref{eq:BL} yields the order $\epsilon^{-2}$ terms 
\begin{subequations}
\begin{align}
\label{eq:O2}
 O(\epsilon^{-2}):\; 0 &= r\partial_{\tilde{q}}\rho_{1,B}^0 + \frac{D_1^2}{2}\partial^2_q\rho_{1,B}^0 \\
\label{eq:O1}
 O(\epsilon^{-1}):\; 0 &= r\partial_{\tilde{q}}\rho_{1,B}^1 + \frac{D_1^2}{2}\partial^2_q\rho_{1,B}^1 
 \end{align}
 \end{subequations}
The solution to the order $\epsilon^{-2}$ equation \eqref{eq:O2}and applying the absorbing boundary condition at $\tilde{q}=0$ yields
\begin{equation}
	\rho_{1,B}^0 = C_1(t) \left (1-\exp\left [-\frac{2r}{D_1^2}\tilde{q}\right]\right ). 
\end{equation}
The order $\epsilon^{-1}$ equation \eqref{eq:O1} has the same solution as above, after applying the absorbing boundary, 
\begin{equation}
	\rho_{1,B}^1 = C_2(t) \left (1-\exp\left [-\frac{2r}{D_1^2}\tilde{q}\right]\right ). 
\end{equation}

Now consider the interval away from the boundary $[O(\epsilon),\infty)$. Let $\rho_{1,A}$ be the density in this region. The equation in this region is
\begin{equation}
	\label{eq:A}
	\partial_t\rho_{1,A} = \frac{r}{\epsilon}\partial_{q}\rho_{1,A} + \frac{D_1^2}{2}\partial_{q}^2 \rho_{1,A} +\delta(q-b)f_{0}(b,t). 
\end{equation}
Let $\rho_{1,A}$ have the asymptotic expansion
$$
\rho_{1,A} = \rho_{1,A}^0+\epsilon\rho_{1,A}^1 + O(\epsilon^2). 
$$
Note that the $\delta$ term acts on $f_0$ which is a function of $\rho_0$. The asymptotic analysis is for $\rho_1$ only, thus the density $\rho_0$ is an order one term. Substituting the expansion into equation \eqref{eq:A} gives the following equations, separated into their orders of $\epsilon$, 
\begin{subequations}
\begin{align}
    \label{eq:OA1}
    O(\epsilon^{-1}): 0 &= r\partial_q \rho_{1,A}^0\\
    \label{eq:OA11}
    O(1): \partial_t\rho_{1,A}^0 &= r\partial_q \rho_{1,A}^{1} +\frac{D_1^2}{2}\partial^2_q\rho_{1,A}^0+\delta(q-b)f_{0}(b,t). 
\end{align} 
\end{subequations}
The order $\epsilon^{-1}$ equation \eqref{eq:OA1} has the solution
\begin{equation}
	\rho_{1,A}^0 = C_3(t). 
\end{equation}
Note that $\rho_{1,A}$ is a density and thus $\rho_{1,A}^0$ must be integrable on $[O(\epsilon),\infty)$. Thus $C_3(t) = 0$ and 
\begin{equation}
	\rho_{1,A}^0 = 0. 
\end{equation}
The order one equation \eqref{eq:OA11}, by substituting in $\rho_{1,A}^0=0$, gives the solution, 
\begin{equation}
\rho_{1,A}^1 = \left \{ \begin{array}{lcr}
C_4(t) & \mbox{for} & O(\epsilon)\leq q< b \\
C_{4}(t)-\frac{1}{r}f_0(b,t), & \mbox{for} & q\geq b 
\end{array}\right . 
\end{equation}
Note that the constant of integration in each interval of $b$ must be the same. Otherwise, the magnitude of the $\delta$ function in \eqref{eq:OA11} would not be correct. The density $\rho_{1,A}^1$ must be integrable which implies that 
\begin{equation}
    C_4(t) = \frac{1}{r}f_0(b,t).  
\end{equation}
It is assumed that the matching between the $A$ and $B$ solutions must occur at the left most edge of the region $[0,O(\epsilon)]$. That is, for values of $q=O(\epsilon^{1/2})$, 
$$ \rho_{1,B}^0(O(\epsilon^{1/2}),t) = \rho_{1,A}^0 (O(\epsilon^{1/2}),t)$$
and 
$$ \rho_{1,B}^1(O(\epsilon^{1/2}),t) = \rho_{1,A}^1 (O(\epsilon^{1/2}),t).$$

The first equation implies that $C_1(t)=0$ and $\rho_{1,B}^0=\rho_{1,A}^0=0$. In the limit as $\epsilon\to 0$ the second equation yields 
\begin{equation}
	C_2(t) = \frac{1}{r}f_0(b,t). 
\end{equation}

Thus the densities are 
\begin{equation}
	\rho_{1}^0 = 0 
\end{equation}
and 
\begin{equation}
	\rho_{1}^1 = \left \{\begin{array}{lr}
					 \frac{1}{r}f_0(b,t) \left (1-\exp\left [-\frac{2r}{D_1^2}\frac{q}{\epsilon}\right]\right ) & 0\leq q\leq O(\epsilon)\\
					 \frac{1}{r}f_0(b,t) & O(\epsilon)\leq q\leq b \\
					 0 & b<q 

					 \end{array} \right . .
\end{equation}
Note the flux of $\rho_1$ at $q=0$ is, 
\begin{equation}
	f_1(0,t) = r\rho_1(0,t) + \frac{D_1^2}{2}\partial_q\left .\rho_1\right |_{(0,t)}. 
\end{equation}
Using the asymptotic expansion yields, 
\begin{equation}
	f_1(0,t) = \frac{D_1^2}{2}\epsilon\left \{\frac{1}{r}f_0(b,t)\left (-\frac{2r}{D_1^2\epsilon}\right )\right \}= f_0(b,t). 
\end{equation}

Thus the Fokker-Planck type equation for $q(t)$ is 
\begin{align}
	\label{eq:limitrho0}
	\partial_t \rho_0 =& -m\partial_q \rho_0 + \left .\frac{D_0^2}{2}\partial_q^2 \rho_0 +\delta(q)f_0\right |_{q=b}, \quad -\infty<q<b, t\geq0, \\
	\label{eq:limitrho1}
	\rho_1 = &0. 
\end{align}
with the following conditions, 
\begin{align}
	\rho_0(b) &=0\\
	\int_{-\infty}^\infty \rho_0(q,t) \;dq &= 1. 
\end{align}

\subsection{Pathwise Convergence}
\label{sec:PathConvergence}

For this section and next, a useful lemma is first stated and proved. 

\begin{lemma}
\label{lemma:decay}
Let $\mathcal{N}_\epsilon(T)$ be the number of rain events for the $q^\epsilon$ process defined in \eqref{eq:qeq}. The probability of that the number of events is $n$ decays exponentially as $n$ tends to 
infinity, i.e. for $0<s<\min\{rb/\epsilon D_2^2, mb/D_1^2\}$
$$ P(\mathcal{N}_\epsilon(T) = N) \leq \exp\left \{sT- \frac{Nmb}{D_1^2}\left (\sqrt{1+\frac{2D_1^2 s}{m^2}} - 1\right )\right \}.$$
\end{lemma}
\begin{proof}
Note that the process
	$\mathcal{N}_\epsilon(T)$ is a renewal process. It is defined by the interarrival times, 
	\begin{equation}
		S_n = \tau_{n}^{d,\epsilon}+\tau_n^{r,\epsilon}, \quad n\geq1, 
	\end{equation}
	where $\tau_{i}^{d,\epsilon} \;(\tau_i^{r,\epsilon})$ is the duration for the $i$th dry (rain) event of the $\sigma^\epsilon$ process.  The distributions of $\tau_i^{d}$ and $\tau_i^{d,\epsilon}$ are the same independent of $\epsilon$, while $\tau_i^{r,\epsilon}$ depends on epsilon. To estimate the sum in \eqref{eq:sumIID}, the probability of having 
	 $N$ rain events in time $T$ is estimated using the Central Limit Theorem. Consider the probability
	 \begin{equation}
	 	P\left (\mathcal{N}^\epsilon(T) = N\right ) = P\left( S_1+S_2+\cdots S_{N} \leq T, \quad S_1+S_2+\cdots +S_N + S_{N+1}>T\right ). 
	\end{equation}
	The probability on the right hand side is estimated crudely by only considering one of the two events. Note that $S_1,S_2,S_3,...,S_n$ are IID random variables with $E[S_1] = E[\tau^{d,\epsilon}+\tau^{r,\epsilon}]$, and $\sigma^2 = \mbox{Var}(S_1)<\infty$, thus	
		\begin{align}
		P(\mathcal{N}^\epsilon(T) = N) \leq P\left( S_1+S_2+\cdots + S_{N} \leq T\right ).
		\end{align}
		The above probability is estimated by using a variant of the Chernoff bound \cite{h94}. That is, 
		\begin{equation}
			P\left( S_1+S_2+\cdots S_{N} \leq T\right ) \leq \exp(sT) \prod_{i=0}^n E[e^{-sS_i}],
		\end{equation}
		for $s>0$, where $E[e^{-sS_i}] = M_{S_i}(s)$ is the moment generating function for the random 
		variable $S_i$. The moment generating function factors due to independence of $\tau_i^{r,\epsilon}$ and $\tau_i^{d,\epsilon}$, 
		$$ M_{S_i}(s) = M_{\tau_i^{r,\epsilon}}(s)M_{\tau_i^{d,\epsilon}}(s).$$
		These moment generating functions are computed explicitly from the distributions found in \cite{hs15siap}. They are, 
		\begin{align}
		M_{\tau_i^{r,\epsilon}} &= \int_0^\infty e^{-st}\rho_r(t)\;dt = \exp\left \{ \frac{-rb}{\epsilon D_2^2}\left (\sqrt{1+\frac{2D_2^2 s \epsilon^2}{r^2}} - 1\right )\right \}\\ 
		M_{\tau_i^{d,\epsilon}} &= \int_0^\infty e^{-st}\rho_d(t)\;dt = \exp\left \{ \frac{-mb}{D_1^2}\left (\sqrt{1+\frac{2D_1^2 s}{m^2}} - 1\right )\right \}.
		\end{align}
		where $s<\min\{rb/\epsilon D_2^2, mb/D_1^2\}$. 
		Thus Chernoff's bound yields, 
		\begin{align}
		P(\mathcal{N}_\epsilon(T) = N) &\leq P\left( S_1+S_2+\cdots S_{n} \leq T\right )\\
		 &\leq \exp(sT) \prod_{i=0}^N E[e^{-sS_i}] \\
		 &= \exp\left \{sT-\frac{Nrb}{\epsilon D_2^2}\left (\sqrt{1+\frac{2D_2^2 \epsilon ^2 s}{r^2}} - 1\right )- \frac{Nmb}{D_1^2}\left (\sqrt{1+\frac{2D_1^2 s}{m^2}} - 1\right )\right \}\\
		 &\leq \exp\left \{sT- \frac{Nmb}{D_1^2}\left (\sqrt{1+\frac{2D_1^2 s}{m^2}} - 1\right )\right \}. \label{eq:expdecay}
		\end{align}
\end{proof}

With this lemma, convergence from $E^\epsilon$ to $E$ is shown in 
$L^2(\Omega)$ with respect to the uniform metric on the space 
of continuous functions $C[0,T]$.

\begin{theorem}
\label{thm:qconv}
Let $q_t^\epsilon$ be defined as 
$$ q_t^\epsilon = E_t^\epsilon+P_t^\epsilon$$ 
where $E_t,P_t$ are solutions to the SDEs in \eqref{eq:EPdef}.  Furthermore let $E_t$ be defined as the solution to 
\eqref{eq:Edef}. Then
\begin{equation} 
\lim_{\epsilon\to 0} E \left [\left ( \sup_{0\leq t\leq T}| E_t^\epsilon - E_t|\right )^2\right ] = 0. 
\end{equation}
\end{theorem}

\begin{proof}
To begin, note that the SDEs for $E^\epsilon$ and $E$ (see \eqref{eq:EPdef}) only differ when $\sigma^\epsilon(t) = 1$. Thus, the solutions to the SDEs give the formula
\begin{equation}
    \left |E^\epsilon(t)-E(t) \right | = \left |\sum_{i=1}^{N^\epsilon(t)} \int_{\tau_i^d}^{\tau_i^d+\tau_i^r} m\;dt + \int_{\tau_i^d}^{\tau_i^d+\tau_i^r} D_0\;dW_t\right |, 
\end{equation}
where $\mathcal{N}^\epsilon(T)$ is the number of rain events for $T<\infty$ and $\epsilon>0$ fixed.
The number
of rain events is conditioned to be $N$. By Lemma~\ref{lemma:decay}, the sum in \eqref{eq:sumIID} converges due to the fast decay of $P(\mathcal{N}_\epsilon(T)=N)$ as $N\to\infty$.
Note that $m>0$ and the stochastic integral is a martingale and 
Doob's maximal inequality yields, 
\begin{equation}
    E\left [\left (\sup_{0\leq t\leq T}\left |E^\epsilon(t)-E(t) \right |\right )^2\right ] \leq \sum_{N=1}^\infty 4E\left [\left .\left |\sum_{i=1}^{N} \int_{\tau_i^d}^{\tau_i^d+\tau_i^r} m\;dt + \int_{\tau_i^d}^{\tau_i^d+\tau_i^r} D_0\;dW_t\right |^2 \right | N^\epsilon(T)=n\right] P(\mathcal{N}^\epsilon(T)=n). 
\end{equation}
Applying the Cauchy-Schwarz inequality to the sum and the 
It\^{o} isometry to the stochastic integral yields
\begin{equation}
 E\left [\left (\sup_{0\leq t\leq T}\left |E^\epsilon(t)-E(t) \right |\right )^2\right ] \leq \sum_{n=1}^\infty C_n E\left [\left . m|\tau^r|^2  + D_0^2|\tau^r| \right | N^\epsilon(T)=n\right] P(N^\epsilon(T)=n). 
\end{equation}
This sum converges due to the fast decay of $P(N_\epsilon(T)=n)$ as shown in Eq.\eqref{eq:expdecay}. 
	Tonelli's theorem allows the limit as $\epsilon\to0$ to exchange with the infinite sum. 
	
To finish the theorem the following moments of $\tau_i^{r,\epsilon}$ are used. The integrals can be computed exactly using the densities for $\tau_i^{r,\epsilon}$ found in \cite{hs15siap}. They are
\begin{equation}
    E[\tau_i^{r,\epsilon}] = \frac{b\epsilon}{r}, \quad E[|\tau_i^{r,\epsilon}|^2] = \frac{bD^2\epsilon^3}{r^3}+\frac{b^2\epsilon^2}{r^2}.
\end{equation}
\end{proof}

\subsection{Distributional Convergence}
\label{sec:MeanSquare}

In this subsection $L^2(\Omega)$ of $\sigma^\epsilon$ to $\sigma$ is shown with respect to a generalized 
function norm. This norm is the one considered
due to the nature of the delta function. It is also
a natural norm to consider as it is an integrated 
error. That is, this norm considers the accumulation of 
errors after running the model for time $T>0$. 

\begin{theorem}
\label{thm:rain}
Let $\phi:[0,\infty)\to\mathbb{R}$ be a test function in $C^\infty_c(0,\infty)$. Let $\sigma(t)$ and $\sigma^\epsilon(t)$ be defined in \eqref{eq:sigma}. Then 
\begin{equation}
    \lim_{\epsilon\to0}E[\left ( |\langle \sigma^\epsilon(t),\phi(t)\rangle -\langle\sigma(t),\phi(t)\rangle|\right )^2] = 0,
\end{equation}
where 
\begin{equation}
    \langle f(t),g(t)\rangle = \int_0^\infty f(t)g(t)\;dt. 
\end{equation}

\begin{proof}
To prove the theorem, the expectation is conditioned
on the number of events $\mathcal{N}^\epsilon(t)$, as is done in the previous section. Thus the 
expectation is
\begin{align}
\label{eq:sumIID}
    &E[\langle |\sigma^\epsilon(t) - \sigma(t)|,\phi(t)\rangle^2] \\
    =& \sum_{N=1}^\infty E\left [\left .\left (\sum_{i=1}^N\int_{\tau_i^{d,\epsilon}}^{\tau_i^{d,\epsilon}+\tau_i^{r,\epsilon}}\sigma^\epsilon(t)\phi(t)\;dt \right . \right . \right .\\
    -& \left .\left .\left . \int_0^T b\delta(t-\tau_i^d)\phi(t)\;dt -\sum_{i=N}^{\mathcal{N}(T)}\int_0^T\sigma(t)\phi(t)\;dt\right )^2\right | \mathcal{N}^\epsilon(T) = N\right ]P(\mathcal{N}^\epsilon(T) = N)
\end{align}
where $\mathcal{N}(T)$ is the number of 
dry events for the $\sigma(t)$ process up to time 
$T$. Again, because of the decay of $P(\mathcal{N}^\epsilon(T) = N)$ as $N\to\infty$ given in Lemma~\ref{lemma:decay}, the infinite sum converges. 

To estimate the quantity in \eqref{eq:sumIID}, one rain event is considered and the Cauchy-Schwarz bound will be used. Consider the $i$th rain event, 
\begin{align}
    \int_{\tau_i^{d,\epsilon}}^{\tau_i^{d,\epsilon}+\tau_i^{r,\epsilon}} \sigma^\epsilon(t)\phi(t)\;dt -b\phi(\tau_i^d) &= \int_{\tau_i^{d,\epsilon}}^{\tau_i^{d,\epsilon}+\tau_i^{r,\epsilon}} \frac{r}{\epsilon}\phi(t)-\frac{r}{\epsilon}\phi(\tau_1^{d,\epsilon})+\frac{r}{\epsilon}\phi(\tau_i^{d,\epsilon})\;dt +b\phi(\tau_i^{d,\epsilon})-b\phi(\tau_i^{d,\epsilon}) -b\phi(\tau_i^d) \\
    =& \int_{\tau_i^{d,\epsilon}}^{\tau_i^{d,\epsilon}+\tau_i^{r,\epsilon}}\frac{r}{\epsilon}\left (\phi(t)-\phi(\tau_i^{d,\epsilon})\right)\;dt + \left (\frac{r}{\epsilon}\tau_i^{r,\epsilon}-
    b\right)\phi(\tau_i^{d,\epsilon})+b(\phi(\tau_i^{d,\epsilon})-\phi(\tau_i^d)). 
\end{align}
The function $\phi(t)$ is smooth on $[0,T]$ and thus is locally Lipschitz. Let the Lipschitz constant be $K>0$. Then, along with the Cauchy-Schwarz inequality,
\begin{align}
   & \left |\int_{\tau_i^{d,\epsilon}}^{\tau_i^{d,\epsilon}+\tau_i^{r,\epsilon}} \sigma^\epsilon(t)\phi(t)\;dt -b\phi(\tau_i^d)\right| \\ 
    \leq & C\int_{\tau_i^{d,\epsilon}}^{\tau_i^{d,\epsilon}+\tau_i^{r,\epsilon}}\frac{r}{\epsilon}K\left |t-\tau_i^{d,\epsilon}\right|\;dt + C\left |\left (\frac{r}{\epsilon}\tau_i^{r,\epsilon}-
    b\right)\phi(\tau_i^{d,\epsilon})\right|+C|\phi(\tau_i^{d,\epsilon}-\phi(\tau_i^{d})|\\
    \leq & C\frac{r}{\epsilon}K|\tau_i^r|^2 + C\left |\left (\frac{r}{\epsilon}\tau_i^{r,\epsilon}-
    b\right)\phi(\tau_i^{d,\epsilon})\right|+C|\phi(\tau_i^{d,\epsilon})-\phi(\tau_i^{d})|.
\end{align}
With the last inequality resulting from $t-\tau_i^{d,\epsilon}$ being an increasing function on $[\tau_i^{d,\epsilon},\tau_i^{d,\epsilon}+\tau_i^{r,\epsilon}]$. 

Using the inequality above, along with the Cauchy-Schwarz inequality, the quantity in \eqref{eq:sumIID} is bounded by
\begin{align}
    \sum_{N=1}^\infty& E\left [\left .\left (\sum_{i=1}^N\int_{\tau_i^{d,\epsilon}}^{\tau_i^{d,\epsilon}+\tau_i^{r,\epsilon}}\sigma^\epsilon(t)\phi(t)\;dt \right . \right . \right .\\
    -&\left . \left . \left .\int_0^T b\delta(t-\tau_i^d)\phi(t)\;dt -\sum_{i=N}^{\mathcal{N}(T)}\int_0^T\sigma(t)\phi(t)\;dt\right )^2\right | \mathcal{N}^\epsilon(T) = N\right ]P(\mathcal{N}^\epsilon(T) = N) \\
    \leq &  \sum_{N=1}^\infty \sum_{i=1}^N\left ( C\left (\frac{r}{\epsilon}\right)^2E[|\tau_i^r|^4] + CE\left [\left |\left (\frac{r}{\epsilon}\tau_i^{r,\epsilon}-
    b\right)\phi(\tau_i^{d,\epsilon})\right|^2\right ]+CE\left [|\phi(\tau_i^{d,\epsilon})-\phi(\tau_i^{d})|^2\right ] \right )P(\mathcal{N}^\epsilon(T) = N)\\
    +& \sum_{N=1}^\infty CE\left [\left (\sum_{i=N+1}^{\mathcal{N}(T)}\int_0^T\sigma(t)\phi(t)\;dt\right )^2\right ]P(\mathcal{N}^\epsilon(T) = N), \label{eq:finalineq}
\end{align}
where all expectations are conditional on $\mathcal{N}^\epsilon(T)=N$. 

To finish the theorem the following moments of $\tau_i^{r,\epsilon}$ are used
\begin{equation}
    E[\tau_i^{r,\epsilon}] = \frac{b\epsilon}{r}, \quad E[|\tau_i^{r,\epsilon}|^2] = \frac{bD^2\epsilon^3}{r^3}+\frac{b^2\epsilon^2}{r^2}, \quad E[|\tau_i^{r,\epsilon}|^4] = \frac{b^4\epsilon^4}{r^4} + 6\frac{b^3D_1^2\epsilon^5}{r^5} + 15 \frac{b^2D_1^4\epsilon^6}{r^6}+15\frac{bD_1^6\epsilon^7}{r^7}.
\end{equation}
Thus the first term in \eqref{eq:finalineq} is 
\begin{equation}
    \left (\frac{r}{\epsilon}\right)^2E[|\tau_i^r|^4] = O(\epsilon^2).
\end{equation}
The second term in \eqref{eq:finalineq} is
\begin{align}
    E\left [\left |\left (\frac{r}{\epsilon}\tau_i^{r,\epsilon}-
    b\right)\phi(\tau_i^{d,\epsilon})\right|^2\right ]  =& E\left [\left (\frac{r}{\epsilon}\tau_i^{r,\epsilon}-
    b\right)^2\right ]E\left [\phi(\tau_i^{d,\epsilon})^2\right ] \\
    =& E\left [\left (\left (\frac{r}{\epsilon}\tau_i^{r,\epsilon}\right )^2 - 2b\frac{r}{\epsilon}\tau_i^{r,\epsilon}+b^2\right )\right ]E\left [\phi(\tau_i^{d,\epsilon})^2\right ]\\
    =& O(\epsilon^3) 
\end{align}
where the expectation turns into a product because $\tau_i^{r,\epsilon}$ and $\tau_i^{d,\epsilon}$ are independent. The third
term of \eqref{eq:finalineq}, $E\left [|\phi(\tau_i^{d,\epsilon})-\phi(\tau_i^{d})|^2\right ]$, is zero because $\tau_i^{d,\epsilon}$ and $\tau_i^d$ have the same distribution. 

For the last ``remainder'' term, it is written as a comparison
between $\mathcal{N}(T)$ and $\mathcal{N}^\epsilon(T)$, 
\begin{align}
    \sum_{N=1}^\infty& CE\left [\left .\left (\sum_{i=N+1}^{\mathcal{N}(T)}\int_0^T\sigma(t)\phi(t)\;dt \right )^2 \right | \mathcal{N}^\epsilon(T) = N \right ]P(\mathcal{N}^\epsilon(T) = N) \\
    =& CE\left [\left (\sum_{i=\mathcal{N}^\epsilon(T)+1}^{\mathcal{N}(T)}\int_0^T\sigma(t)\phi(t)\;dt \right )^2\right ] \\
    = & \sum_{N<M} E\left [\left .\left (\sum_{i=N+1}^{M}\int_0^T\sigma(t)\phi(t)\;dt \right )^2\right | \mathcal{N}^\epsilon(T)= N<\mathcal{N}(T)=M\right ]P(\mathcal{N}^\epsilon(T)<\mathcal{N}(T))
\end{align}
If $\mathcal{N}^\epsilon(T)<\mathcal{N}(T)$, then the processes $E_t^\epsilon$ and $E_t$ must be at least $b$ units apart. Thus
\begin{equation}
P(\mathcal{N}^\epsilon(T)<\mathcal{N}(T))\leq P(|E^\epsilon(t)-E(t)|>b). 
\end{equation}
From theorem \ref{thm:qconv}, this quantity tends to zero as 
$\epsilon\to0$. 

Thus, the sum in \eqref{eq:sumIID} converges due to the fast decay of $P(N_\epsilon(T)=n)$ as shown in Eq.\eqref{eq:expdecay}. 
Tonelli's theorem allows the limit as $\epsilon\to0$ to exchange with the infinite sum. This along with Theorem \ref{thm:qconv} completes the proof. 
\end{proof}

\end{theorem}

\section{Conclusions}
\label{sec:conclusions}

In this paper a threshold model for moisture and rain was shown to converge to 
interesting processes for various modes of convergence. The original threshold model processes, defined in equation \eqref{eq:D2}, 
originated from \cite{sn14} and were studied in \cite{hs15siap}. There, exact formulas were 
derived for various quantities of interest such as stationary distributions and expected rainfall. In \cite{asn16} a connection was made between the rain process from the threshold model and the point process defined in equation \eqref{eq:sigma}. Here convergence for the moisture processes were defined and shown for the Fokker-Planck equation as well as the paths of the processes. Furthermore, the convergence of the rain process were shown in mean square difference with respect to the space of generalized functions. 

Using a point process to approximate rainfall allows simplification for computation and exact formulas. For example, the autocorrelation function is known exactly in the case 
of point processes as shown in \cite{asn16}. Furthermore, point processes have been studied
extensively in the neural science literature \cite{sg13} and many statistics have been 
derived. 

The proofs shown here are revealing on their own. The Fokker-Planck derivation in Section \ref{sec:FP} shows the density for the moisture in the rain state tends to zero while
the flux term remains allowing for the ``teleporting'' boundary condition that arises for
limiting moisture process. For the convergence of paths of moisture shown in Theorem \ref{thm:qconv} the moisture process must first be decoupled into a moistening and precipitating process. Then the moistening process is shown to converge (Theorem \ref{thm:qconv}) 
while the precipitating process contains all of the discontinuities. Finally, the proof of convergence of 
the rain processes in Theorem \ref{thm:rain} gives estimates that would be useful for 
determining the error rates for using the point process approximation. 


\bigskip

{\it Acknowledgements.}
The research of
author Hottovy is partially supported by the National
Science Foundation under Grant DMS-1815061.

\bibliographystyle{apalike}
\bibliography{weather_bib}

\end{document}